\documentclass[leqno,a4paper]{article}
\usepackage{amsmath,bbm, amsthm,amssymb}
\usepackage{eufrak}
\newcommand{\bq}{\mathcal{Q}}
\newcommand{\bn}{\mathcal{N}}
\newcommand{\mF}{\mathcal{F}}
\newcommand{\mO}{\mathcal{O}}
\newcommand{\ZZ}{{\bf Z}}
\newcommand{\C}{{\bf C}}
\newcommand{\mL}{\mathcal{L}}
\newcommand{\VV}{\mathbbm{V}}
\newcommand{\ba}{\boldsymbol{a}}
\newcommand{\be}{\boldsymbol{e}}
\newcommand{\bu}{\boldsymbol{u}}
\newcommand{\bv}{\boldsymbol{v}}
\newcommand{\bw}{\boldsymbol{w}}
\newcommand{\ff}{\boldsymbol{f}}
\newcommand{\bg}{\boldsymbol{g}}
\newcommand{\bG}{\boldsymbol{G}}
\newcommand{\bC}{\boldsymbol{C}}
\newcommand{\E}{\boldsymbol{E}}
\newcommand{\U}{\boldsymbol{U}}
\newcommand{\mB}{\mathcal{B}}
\newcommand{\cM}{\mathcal{M}}
\newcommand{\cH}{\mathcal{H}}
\newcommand{\mT}{\mathcal{T}}
\newcommand{\mP}{\mathcal{P}}
\newcommand{\N}{{\bf N}}
\newcommand{\NN}{{\mathbb N}}
\newcommand{\R}{{\mathbb R}}
\newcommand{\ve}{{\overline{e}}}
\newcommand{\ou}{{\overline{u}}}
\newcommand{\ox}{{\overline{x}}}
\newcommand{\oy}{{\overline{y}}}
\newcommand{\oq}{{\overline{q}}}
\newcommand{\ov}{{\overline{v}}}
\newcommand{\oR}{{\overline{\R^n_+}}}
\newcommand{\on}{{\overline{\nu}}}
\newcommand{\ux}{{\underline{x}}}
\newcommand{\uy}{{\underline{y}}}
\newcommand{\HH}{\mathbbm{H}}
\newcommand{\De}{{\Delta}}
\newcommand{\na}{{\nabla}}
\newcommand{\pa}{{\partial}}
\newcommand{\cn}{{\nu}}
\newcommand{\ro}{{\rho}}
\newcommand{\ta}{{\tau}}
\newcommand{\la}{{\lambda}}
\newcommand{\de}{{\delta}}
\newcommand{\om}{{\omega}}
\newcommand{\Om}{{\Omega}}
\newcommand{\Ov}{{\overline{\Omega}}}
\newcommand{\Ga}{{\Gamma}}

\newcommand{\ua}{{\underline{a}}}
\newcommand{\ub}{{\underline{b}}}
\newcommand{\ul}{{\underline{l}}}
\newcommand{\n}{{\underline{n}}}
\newcommand{\ui}{{\underline{i}}}
\newcommand{\al}{{\alpha}}
\newcommand{\te}{{\theta}}
\newcommand{\bt}{{\beta}}
\newcommand{\utau}{{\underline{\tau}}}
\newcommand{\ugamma}{{\underline{\gamma}}}
\newcommand{\z}{{\zeta}}
\newcommand{\ci}{{\chi}}
\newcommand{\m}{{\mu}}
\newcommand{\ep}{{\epsilon}}
\newcommand{\com}{{\widetilde{\omega}}}
\newcommand{\cphi}{{\widetilde{\phi}}}
\newcommand{\cf}{{\widetilde{f}}}
\newcommand{\cu}{{\widetilde{u}}}

\newcommand{\nn}{\boldsymbol{n}}
\newcommand{\ch}{{\widehat{h}}}
\newcommand{\hu}{{\widehat{u}}}
\newcommand{\hw}{{\widehat{w}}}
\newcommand{\hb}{{\widehat{b}}}
\newcommand{\cv}{{\widehat{v}}}
\newcommand{\cb}{{\widehat{b}}}
\newcommand{\cK}{{\mathcal{K}}}
\newcommand{\K}{{\bf K}}
\newcommand{\Pf}{{\Phi}}
\newcommand{\pf}{{\phi}}
\newcommand{\ps}{{\psi}}

\newcommand{\bB}{\boldsymbol{B}}
\newcommand{\bD}{\boldsymbol{D}}
\def\dy{\displaystyle}
\def\vs1{\vspace{1ex}}

\newtheorem{theorem}{Theorem}[section]
\newtheorem{lemma}[theorem]{Lemma}
\newtheorem{proposition}{Proposition}[section]
\newtheorem{assumption}{Assumption}[section]

\newtheorem{problem}{Problem}[section]
\numberwithin{equation}{section}
\def\be{\begin{equation}}
\def\ee{\end{equation}}
\def\ed{\end{document}}
\begin{document}
\title{\bf\normalsize On some regularity results for $\,2-D\,$ Euler equations and linear elliptic b.v. problems.}
\author{by H.~Beir\~ao da Veiga}
\maketitle
\textit{\phantom{aaaaaaaaaaaaaaaaaaaaaaaaaa}}
%
\bigskip

\begin{abstract}
About thirty years ago we looked for "minimal assumptions" on the
data which guarantee that solutions to the $\,2-D\,$ evolution Euler
equations in a bounded domain are classical. Classical means here
that all the derivatives appearing in the equations and boundary
conditions are continuous up to the boundary. Following a well known
device, the above problem led us to consider this same regularity
problem for the Poisson equation under homogeneous Dirichlet
boundary conditions. At this point, one was naturally led to
consider the extension of this last problem to more general linear
elliptic boundary value problems, and also to try to extend the
results to more general data spaces. At that time, some side results
in these directions remained unpublished. The first motivation for
this note is a clear description of the route followed by us in
studying these kind of problems. New results and open problems are
also considered.

\vspace{.2cm}

\,35J25,\,35Q30.

\vspace{.2cm}

\end{abstract}
\bibliographystyle{amsplain}
\section{Introduction.}\label{inttas}%
In this note the central problem is the following. Looking for
"minimal assumptions" on the data which guarantee that solutions to
a specific stationary, or evolution, problem are classical. This
problem is called here the \emph{minimal assumptions problem}. We
say that solutions are \emph{classical} if all derivatives appearing
in the equations are continuous up to the boundary on their domain
of definition.\par%
The starting point of these notes is reference \cite{BVJDE}, where
the main goal was to look for \emph{minimal assumptions} on the data
which guarantee \emph{classical} solutions to the $\,2-D\,$ Euler
equations in a bounded domain
\begin{equation}
\label{alho}%
\left\{
\begin{array}{l}\vspace{0.5ex}
\partial_t\,\bv +\,(\bv\cdot\,\nabla)\,\bv+\,\nabla\,\pi=\,\ff
\quad \textrm{in} \quad Q\equiv \,\R^+ \times \Om\,,\\
{div}\,\bv=\,0 \quad \textrm{in} \quad Q\,;\\
\bv \cdot \,\n=\,0 \quad \textrm{on} \quad \R \times \Ga \,; \\%
\bv(0)=\,\bv_0 \quad \textrm{in} \quad \Om\,,
\end{array}
\right.%
\end{equation}
where here, and everywhere in the sequel, the initial data
$\,\bv_0\,$ is assumed to be divergence free and tangential to the
boundary. Merely for simplicity, every time we are concerned with
$\,2-D\,$ Euler equations, we assume that the domain $\,\Om\,$ is
simply connected. Recall that in the $\,2-D\,$ case the
curl of the velocity can be, and is here, identified with a scalar field.\par%
We start by briefly explaining why the curl of the velocity plays a
central role in equation \eqref{alho}. This follows from a quite
favorable situation, fruit of the lucky combination of two distinct
facts. The first one is that the second and the third equations in
the well known elliptic system \eqref{naosei2}, see below, are still
included in \eqref{alho}. Hence, at each fixed time, the solution
$\,\bv(t)\,$ of problem \eqref{alho} is completely determined by the
\emph{scalar} quantity $\,{curl}\,\bv(t)\,$. The second favorable
feature is that, as a rule, this last quantity  "is transported by
the characteristics" (assume, for simplicity, that $\ff=\,0\,$).\par%
The above setup led us, in reference \cite{BVJDE}, to look for data
spaces $\,C_*(\,\Ov)\,,$ as large as possible, for which the two
assumptions below hold. In the second assumption, a loss of
regularity going from the curl to the gradient is deliberately
allowed since in the minimal assumptions problem nothing more than
continuity is required to $\,\na\,\bv\,$. Roughly speaking, getting
more than continuity, could mean that assumptions on the data are
not "minimal".
\begin{assumption}
The space $\,C_*(\Ov)\,$ satisfies the following property. If
$\,{curl}\,\bv_0 \in \,C_*(\Ov)\,$, and $\,{curl}\,\ff \in
\,L^1(\,\R^+;\, C_*(\Ov)\,)\,$, then the global solution $\,\bv\,$
of problem \eqref{alho} satisfies
$\,{curl}\,\bv \in\,C(\,\R^+\,;\,C_*(\Ov)\,)\,.$%
\label{ppdois}
\end{assumption}
\begin{assumption}
The space $\,C_*(\Ov)\,$ satisfies the following property. If $\,\te
\in \,C_*(\Ov)\,,$ then $\,\na\, \bv \in\,\C(\Ov)\,,$ and
$\,\|\,\na\,\bv\,\| \leq \,c_0\,\|\,\te\,\|_*\,,$ where $\,\bv\,$ is
the solution to the problem \eqref{naosei2}. So, for divergence free
vector fields, tangential to the boundary, the estimate $\,
\|\,\na\,\bv\,\| \leq \,c_0\,\|\,{curl}\,\bv\,\|_*\,$
holds.%
\label{laplaces22-cestar}
\end{assumption}
It is worth noting that in this section we proceed  as if we have
not already found a suitable space $\,C_*(\Ov)\,.$ We are just
setting up the problem. We believe this helps the understanding of
our approach. Following this idea, we introduce our specific
functional space $\,C_*(\Ov)\,$ only in the next section.\par%
\emph{If} a functional space $\,C_*(\Ov)\,$ satisfies the two above
assumptions, the following result is immediate.
\begin{proposition}
Assume that the functional space  $\,C_*(\Ov)\,$ enjoys the above
assumptions \ref{ppdois} and \ref{laplaces22-cestar}. Furthermore,
let $\,{curl}\,\bv_0 \in \,C_*(\Ov)\,$, and $\,{curl}\,\ff \in
\,L^1(\,\R^+;\, C_*(\Ov)\,)\,$. Then, the global solution $\,\bv,\,$
to problem \eqref{alho} is classical:
\begin{equation}
\na\,\bv \in\,\C(\,\R^+\,;\,C(\Ov)\,)\,.%
\label{zeuler-cestar-bis}
\end{equation}
\label{eulas-cestar-bis}
\end{proposition}%
Clearly, the above two assumptions and proposition hold for our
specific space $\,C_*(\Ov)\,$. See Theorems \ref{eulaszeta},
\ref{rotlaplaces}, and \ref{eulas}.\par%
Assumptions \ref{ppdois} and \ref{laplaces22-cestar} alone are
clearly insufficient to determine  good choices of spaces
$\,C_*(\Ov)\,$, since arbitrarily high regularity to the solutions
is not imposed. Actually, our interest in the minimal assumption
problem comes from the particular situation in which data are
H\"older continuous. In fact, the above two assumptions hold by
setting $\,C_*(\,\Ov)=\,C^{0,\,\la}(\,\Ov)\,.$ However a H\"older
continuity assumption on the data is unnecessarily restrictive,
since it adds too much regularity to our continuity requirement, see
\eqref{mates}. On the other hand, the choice
$\,C_*(\,\Ov)=\,C(\,\Ov)\,$ is too wide. In this case assumption
\ref{laplaces22-cestar} holds, however assumption \ref{ppdois} is
false in this setting. In conclusion, a basic problem in reference
\cite{BVJDE} was to single out functional spaces $\,C_*(\,\Ov)\,$
which satisfy assumptions \ref{ppdois} and \ref{laplaces22-cestar},
and for which the embeddings
\begin{equation}
\quad C^{0,\,\la}(\Ov) \subset \,C_*(\Ov)\, \subset \,C(\Ov)\,,%
\label{oito}
\end{equation}
are strict. In conclusion, a main problem was the following.
\begin{problem}
Look for Banach spaces $\,C_*(\Ov)\,,$ which strictly satisfy the
inclusions \eqref{oito}, and enjoy assumptions \ref{ppdois} and
\ref{laplaces22-cestar}.%
\label{oproblema}
\end{problem}%
In reference \cite{BVJDE} we singled out a specific functional space
$\,C_*(\Ov)\,$ which enjoys the three requirements stated in problem
\ref{oproblema}. We postpone the definition of our specific space
$\,C_*(\Ov)\,$ to section \ref{one} below. Obviously, there may
exist other significant functional spaces satisfying the required
properties.%

\vspace{0.2cm}

In the next sections we turn back to the effective resolution of the
above, and related, problems. In fact, study and resolution of the
above problems opens the way to new problems. First of all, problems
related to the \emph{minimal assumptions problem} for more general
elliptic boundary value problems. In \cite{BVJDE}, the minimal
regularity problem for the elliptic system \eqref{naosei2} was
confined to a similar regularity problem for equation
\eqref{lapsim}. Theorem \ref{rotlaplaces} was sufficient for our
purposes. However, at that time, as remarked in \cite{BVJDE}, we had
proved an extension of this result to more general elliptic boundary
value problems (the proof remained unpublished, even though we were
not able to find it in the current literature). Further, in a recent
paper, we extended the proof to the stationary Stokes system, see
Theorem \ref{teoum} below. Similar results hold for more general
linear elliptic problems, as the reader may verify, since the proof
depends only on the behavior of the associated Green's functions.\par%
An interesting research field is the extension of the results to
larger data spaces. In the sequel we refer to two distinct possible
extensions. Unfortunately, we merely obtain partial extensions. As
an example of this situation, compare Theorem \ref{teoum} with
Theorem \ref{prop2}, where continuity is replaced by boundedness.\par%
Finally, partial extensions of theorems \ref{eulaszeta} and
\ref{eulas} to initial data in a functional space $\, B_*(\Ov)\,,$
which strictly contains $\,C_*(\,\Ov)\,$, are shown in section
\ref{bbesp}, see theorems \ref{zeulerin} and \ref{zeulerindos}.

\vspace{0.2cm}

\emph{Plan of the paper:}\par%
In section \ref{one} we recall definition and properties of the real
space $\,C_*(\,\Ov)\,$ introduced in reference \cite{BVJDE}.\par%
In section \ref{conhecidos} we recall some results on Euler
equations and elliptic problems with $\,C_*(\Ov)\,$ data.\par%
In section \ref{doisemeio} we consider Stokes, and other elliptic
problems, with data in $\,C_*(\Ov)\,$. In particular, we show the
connection between problems \eqref{naosei2} and \eqref{lapsim}.\par%
In section \ref{two} we introduce new data spaces, denoted
$\,B_*(\Ov)\,$ and $\,D_*(\Ov)\,,$ satisfying the inclusions
$$
\,C_*(\Ov) \subset\,\,B_*(\Ov) \subset \,D_*(\Ov)\,.
$$
We also introduce a new family of Banach spaces, denoted
$\,D^{0,\,\al}(\Ov)\,,$ a kind of weak extension of the classical
family of H\"older spaces.\par%
In section \ref{thres} the full aim would be to extend the results
with data in $\,C_*(\Ov)\,$ to data in the new spaces $\,B_*(\Ov)\,$
and $\,D_*(\Ov)\,.$ Some partial extension results are shown for
solutions to the Stokes equations. Second order derivatives of
solutions are bounded but, possibly, not continuous. The proofs
depend essentially on the properties of the related Green's functions.\par%
In section \ref{bbesp} we try to extend the results proved in
reference \cite{BVJDE} for the Euler equations with data in
$\,C_*(\,\Ov)\,$, to data in $\,B_*(\,\Ov)\,.$ The extension
obtained is partial, since continuity is replaced by boundedness,
and external forces vanish.%

\vspace{0.2cm}

\emph{Acknowledgement}. Reference \cite{BVJDE} was partially
prepared when the author was  Professor at the Mathematics
Department and the ``Mathematics Research Center'', University of
Wisconsin-Madison,(October 1981--March 1982), and at the
University of Minnesota -- Minneapolis (March 1982 -- June 1982).%
The above paper appears first in the 1982 reference \cite{BVMRC}.
The author is grateful to Professor Robert E.L. Turner for the above
invitation to Madison. The author would also like to take this
occasion to thank Bob for the continuous help in correcting the
English of many papers, together with mathematical advice and
remarks.


\vspace{0.3cm}

\section{The real space $\,C_*(\Ov)\,.$}\label{one}%
In this section, following \cite{BVJDE}, we introduce a specific
space $\,C_*(\Ov)\,$ which enjoys the three requirements stated in
problem \ref{oproblema}. Furthermore, we recall some of the main
properties of this space. For complete proofs see \cite{BVSTOKES}.\par%
We start with some notation. In the following $\Om$ is an open,
bounded, connected set in $\R^n\,$, locally situated on one side of
its boundary $\,\Ga\,.$  The boundary $\Ga$ is of class
$\,C^{2,\,\la}\,,$ for some $\,\la\,,$ $\,0<\,\la \leq \,1\,.$ In
considering the Euler equations it is always assumed that
$\,n=\,2\,,$ a crucial assumption here. On the contrary, in
considering the Stokes equations we assume, merely to simplify index
notation, that $\,n=\,3\,.$ Proofs immediately apply to the $n-$dimensional case.\par%
By $\,C(\Ov)\,$ we denote the Banach space of all real continuous
functions $\,f\,$ defined in $\,\Ov\,$. The classical norm in this
space is denoted by $ \|\,f\,\|\,. $ We also appeal to the classical
spaces $\,C^k(\Ov)\,$ endowed with the standard norm $
\|\,u\,\|_k\,,$ and to the H\"older spaces $\,C^{0,\,\la}(\Ov)\,,$
endowed with the standard norm $\,\|\,f\,\|_{0,\,\la}\,.$ In
particular $\,C^{0,\,1}(\Ov)\,$ is the space of Lipschitz continuous
functions in $\,\Ov\,.$ Further, $\,C^{\infty}(\Ov)\,$ denotes the
set of all
restrictions to $\Ov\,$ of indefinitely differentiable functions in  $\R^n\,$.\par%
Boldface symbols refer to vectors, vector spaces, and so on.
Components of a generic vector $\,\bv\,$ are indicated by $\,v_i\,$,
with  similar notation for tensors. Norms in function spaces, whose
elements are vector fields, are defined in the usual way by means of
the corresponding norms of the components.\par%
We denote by $\,\boldsymbol{e}_i\,$,  $\,i=\,1,\,2,\,3\,,$ the three
Cartesian coordinate unit vectors in  $\R^3\,$. When considering the
$\,2-D$ Euler equations, the planar motion is that generated by the
couple
$\,\boldsymbol{e}_1\,$, $\,\boldsymbol{e}_2\,$.\par%
The symbols $c,\,c_0\,,c_1,\,...,$ denote positive constants
depending at most on $\Om\,$ and $\,n\,.$ We may use the same symbol
to denote different constants.

\vspace{0.2cm}

Next, we define $\,C_*(\Ov)\,,$ and recall some properties of this
functional space. Set
$$
I(x;\,r)=\,\{\,y:\,|y-\,x| \leq\, r\,\}\,, \quad\,\Om(x;\,r)=\, \Om
\,\cap\,I(x;\,r)\,, \quad  \Om_c (x;\,r)=\,\Om-\,\Om(x;\,r)\,.
$$\,.

For $\,f \in \,C(\Ov)\,$ we define, for each $\,r>\,0\,,$
\begin{equation}
\om_f(r) \equiv \, \sup_{\,x,\,y \in\,\Om(x;\,r)} \,|\,f(x)-\,f(y)\,|\,.%
\label{cinco}
\end{equation}
In \cite{BVJDE} we introduced the semi-norm
\begin{equation}%
[\,f\,]_* =\,[\,f\,]_{*,\,\de} \equiv \int_0^\ro \,\om_f(r) \,\frac{dr}{r}\,.%
\label{seis}
\end{equation}
The finiteness of the above integral is known as \emph{Dini's
continuity condition}, see \cite{gilbarg}, equation (4.47). In this
reference, problem 4.2, it is remarked that if $\,f\,$ satisfies
Dini's condition in the whole space $\,\R^n\,$, then its Newtonian
potential is a $\,C^2\,$ function in $\,\R^n\,$.\par%
We define
\begin{equation}
C_*(\Ov) \equiv\,\{\,f \in\,C(\,\Ov): \,[\,f\,]_*
<\,\infty\,\}\,.%
\label{cstar}
\end{equation}
A norm is introduced by setting
$$
\,\|\,f\,\|_{*,\,\de}\equiv\,[\,f\,]_{*,\,\de}+\,\|\,f\,\|\,.
$$
Since
\begin{equation}%
[\,f\,]_{*,\,\rho_1} \leq\,[\,f\,]_{*,\,\rho_2}
\leq\,[\,f\,]_{*,\,\rho_1}+\,2\,\Big(\,\log{\frac{\ro_2}{\ro_1}}\Big)\,\|\,f\,\|\,,
\label{igual}
\end{equation}
for $\,0<\,\ro_1 <\,\ro_2\,,$ norms are essentially independent of
$\,\ro\,.$\par%
Note that by setting
\begin{equation}
\om_f(x;\,r)= \, \sup_{y \in\,\Om(x;\,r)}\,|\,f(x)-\,f(y)\,|\,,%
\label{vinte}
\end{equation}
it follows that
\begin{equation}%
[\,f\,]_* =\,\int_0^\ro  \,\sup_{ \,x \in\,\Ov}\, \om_f(x;\,r)\,
\,\frac{dr}{r}\,.%
\label{catriz}
\end{equation}

\vspace{0.2cm}

The main properties of $\,C_*(\Ov)\,$ are the following.
\begin{theorem}
$\,C_*(\Ov)\,$ is a Banach space.%
\label{propum}
\end{theorem}
\begin{theorem}
The embedding $\, C_*(\Ov) \subset \,C(\Ov)\,$ is compact.%
\label{propum2}
\end{theorem}
\begin{theorem}
The set $\,C^{\infty}(\Ov)\,$ is dense in $\,C_*(\Ov)$ \,.%
\label{lemum}
\end{theorem}
It is worth noting that the crucial property required for the space
$\,C_*(\Ov)\,$ in the proofs of Theorems \ref{laplaces} and
\ref{teoum}, is Theorem \ref{lemum}. This theorem is proved, see
\cite{BVSTOKES}, by appealing to the well known mollification
technique. This density result \emph{up to the boundary} requires a
previous, suitable extension, of the functions outside $\,\Ov\,.$
The following result holds.
\begin{theorem}
Set $\,\Om_{\de} \equiv\,\{\,x:\, dist(x,\,\Om\,) <\,\de\,\}.\,$
There is a $\,\de >0\,$ such that the following statement holds.
There is a linear continuous map $\,T\,$ from $\,C(\Ov)\,$ to
$\,C(\Ov_{\de})\,,$ such that its restriction to $\,C_*(\Ov)\,$ is
continuous from $\,C_*(\Ov)\,$ to $\,C_*(\Ov_{\de})\,,$ and
$\,T\,f\,,$ restricted to $\,\Ov\,,$ coincides with $\,f\,.$%
\label{bah}
\end{theorem}
\section{Some results on Euler equations and elliptic problems
with  $\,C_*(\Ov)\,$ data.}\label{conhecidos}
In this section we refer back to the presentation shown in section
\ref{inttas} (the reader is supposed to have it in mind). The space
$\,C_*(\Ov)\,$ is here no more "abstract", but that defined in
section \ref{one}. Motivation was discussed in section \ref{inttas},
in particular the justification of the single statements (below, in
rigorous form).%

\vspace{0.2cm}

As already recalled, in reference \cite{BVJDE} we have considered
the problem of minimal assumptions on initial data and external
forces sufficient to obtain \emph{classical solutions} of the
$\,2-D\,$ Euler equations \eqref{alho}. According to our definition,
classical solution of this system means here that $\, \nabla\,\bv
\,,\,\partial_t\,\bv\,,\, \nabla\,\pi \in\,C(\,\overline{Q}\,)\,.$\par%
By considering $\,C(\Ov)\,$ as the curl's data space, one has the
following result, proved in \cite{BVJDE}, Theorem 1.1.
\begin{theorem}
Let a divergence free vector field $\,\bv_0\,$, tangent to the
boundary, satisfy $\,{curl}\,\bv_0 \in \,C(\Ov)\,,$ and let
$\,{curl}\,\ff \in \,L^1(\,\R^+;\, C(\Ov)\,)\,.$ Then, the problem
\eqref{alho} is uniquely solvable in the large,
\be%
\label{rotoras-cont}
{curl}\,\bv \in\,C(\,\R^+\,;\,C(\Ov)\,)\,,%
\ee%
and the estimate
\begin{equation}
\|\,{curl} \,\bv(t)\,\| \leq \,\|\,{curl} \,\bv_0\,\|+\,\int_{0}^{t}
\, \|\,{curl} \,\ff(\tau)\,\| \, d\tau %
\label{defquiv}
\end{equation}
holds. If $\,{curl}\ff=\,0\,,$ then $\|\,{curl}\bv(t)\,\|=\,\|\,{curl}\bv_0\,\|\,.$%
\label{ppum}
\end{theorem}
The next step was to replace $\,{curl}\,\bv\,$ by $\,\na\,\bv\,$ in
the left hand side of equation \eqref{defquiv}. The starting point
here is that the two last equations in the elliptic system
\begin{equation}
\left\{
\begin{array}{l}
{curl}\,\bv=\,\te \quad \textrm{in} \quad \Om \,,
\\
div \,\bv=\,0 \quad \textrm{in} \quad \Om \,,
\\
\bv \cdot\,\n=\,0 \quad \textrm{on} \quad \Ga \,, %
\end{array}
\right. \label{naosei2}
\end{equation}
are still included in \eqref{alho}. Furthermore, it is well known
that solutions $\,\bv\,$ of problem \eqref{naosei2} are completely
determined here by the \emph{scalar} quantity $\,\te\,$ (recall the
assumption $\,\Om\,$ simply-connected). Hence, to be allowed to
replace $\,{curl}\,\bv\,$ by $\,\na\,\bv\,$ in the left hand side of
equation \eqref{defquiv} it is sufficient to show that solutions
$\,\bv\,$ of problem \eqref{naosei2} satisfy the estimate $
\,\|\,\na\,\bv\,\| \leq\,c \,\|\,\te\,\|\,.$ Unfortunately, this is
known to be false. In other words, the data space $ \,C(\Ov)\,$ is
too wide. On the other hand, H\"older spaces are here too narrow. In
fact, in this case (see \cite{kato}, \cite{schaef}, and also
\cite{bardos}, \cite{judo}), the above device works, since solutions
of \eqref{naosei2} satisfy the estimate
\begin{equation}
\,\|\,\na\,\bv\,\|_{0,\,\la} \leq\,c \,\|\,\te\,\|_{0,\,\la}\equiv\,
c \,\|\,{curl}\,\bv\,\|_{0,\,\la}\,. \label{mates}
\end{equation}
However this estimate is unnecessarily strong in the context of our
"minimal assumptions problem". So our problem was to find a
functional space $\,C_*(\Ov)\,,$ which satisfies assumptions
\ref{ppdois} and \ref{laplaces22-cestar}, and such that the
inclusions \eqref{oito} are strict. The space $\,C_*(\Ov)\,$ defined
in section \ref{one}, satisfies these requirements. Concerning
Assumption \ref{ppdois}, we have proved in \cite{BVJDE} (see Lemma
4.4 in this reference) the following statement, may be the main
result in the above paper. It is worth noting that the presence of
quite general external forces leads to very substantial, additional
difficulties in the proof of this result.
\begin{theorem}
Let $\,C_*(\Ov)\,$  be the Banach space defined in section
\ref{one}. Assume that $\,{curl}\,\bv_0 \in \,C_*(\Ov)\,$ and
$\,{curl}\,\ff \in \,L^1(\,\R^+;\, C_*(\Ov)\,)\,$. Then, the curl of
the global solution $\,\bv\,$ of problem \eqref{alho} satisfies
$$
{curl}\,\bv \in\,C(\R^+ ; \, C_*(\,\Ov)\,)\,.
$$
Moreover
\begin{equation}
\|\,{curl}\,\bv(t)\,\|_{*} \leq \, e^{\,c_1\,B\,t\,}
\,\Big(\,3\,B+\,[\,{curl}\,\bv_0\,]_{*}+\,[\,{curl}\,\ff\,]_{
\,L^1(0,\,t;\, C_*(\Ov)\,)\,}\,\,\Big)\,,%
\label{estmeulerzeta}
\end{equation}%
where
\begin{equation}
B=\,\|\,{curl}\,\bv_0\,\| +\,\|\,{curl}\,\ff\,\|_{\,L^1(0,\,t;\,
C(\Ov)\,)\,}\,.%
\label{bbeesze}
\end{equation}%
\label{eulaszeta}
\end{theorem}
In \cite{BVJDE} the reader will find the notation
$\z=\,{curl}\,\bv\,,$ $\, \phi=\,{curl}\,\ff\,,$ and
$\z_0=\,{curl}\,\bv_0\,.$\par%
The above result, in the simpler case in which external forces
vanish, has been rediscovered, later on, by other authors.

\vspace{0.2cm}

Together with the above result we have proved that, under the same
assumptions on the data, also Assumption \ref{laplaces22-cestar}
holds. The following result was claimed in reference \cite{BVJDE}.
\begin{theorem}
Let $\,C_*(\Ov)\,$  be the Banach space defined in section
\ref{one}. Let $\,\te \in \,C_*(\Ov)\,,$ and $\,\bv\,$ be the
solution of problem \eqref{naosei2}. Then $\,\na\, \bv \in\,
C(\Ov)\,,$ and $\,\|\,\na\,\bv\,\| \leq \,c_0\,\|\,\te\,\|_*\,.$ In
other words, for divergence free vector fields, tangent to the
boundary, the estimate  $ \|\,\na\,\bv\,\| \leq
\,c_0\,\|\,{curl}\,\bv\,\|_*\,$ holds.%
\label{rotlaplaces}
\end{theorem}
In \cite{BVJDE}, following a classical device, the study of system
\eqref{naosei2} was confined to a similar regularity problem for
equation \eqref{lapsim}. For convenience, we postpone this
point to section \ref{doisemeio}.\par%
Theorems \ref{eulaszeta} and \ref{rotlaplaces} yield the following
statement (Theorem 1.4 in  \cite{BVJDE}).
\begin{theorem}
Let $\,C_*(\Ov)\,$  be the Banach space defined in section
\ref{one}. Further, let $\,{curl}\,\bv_0  \in \,C_*(\Ov)\,$ and
$\,{curl}\,\ff \in \,L^1(\,\R^+;\, C_*(\Ov)\,)\,.$ Then, the global
solution $\,\bv\,$ to problem \eqref{alho} is continuous in time
with values in $\,\C^1(\Ov)\,,$
\begin{equation}
\bv \in\,C(\,\R^+\,;\,\C^1(\Ov)\,)\,.%
\label{zeuler}
\end{equation}%
Furthermore, the estimate
\begin{equation}
\|\,\bv(t)\,\|_{\C^1(\Ov)} \leq\,c\,e^{\,c_1\,B\,t\,}
\,\{\,\|\,{curl}\,\bv_0\,\|_{C_*(\Ov)}+\,\|\,{curl}\,\ff\,\|_{
\,L^1(0,\,t;\, C_*(\Ov)\,)\,}\,\}%
\label{estmeuler}
\end{equation}%
holds for all $\,t\in \R^+\,,$ where $B$ is given by \eqref{bbeesze}.\par%
Moreover, $\,\partial_t\,\bv\,$ and $\,\nabla\,\pi\,$ are continuous
in $\,\overline{Q}\,$ if both terms $\,\ff_0\,$ and $\,\na\,F\,,$ in
the canonical Helmholtz decomposition $\,\ff=\,\ff_0 +\,\na\,F\,$
satisfy, separately, this same continuity property. So, all
derivatives that appear in equations \eqref{alho} are continuous
in $\,\overline{Q}\,$ (classical solution).%
\label{eulas}
\end{theorem}
If $\,\Om\,$ is not simply connected the results still apply. See
the appendix 1 in \cite{BVJDE}.\par%

\vspace{0.2cm}

It is worth noting that the proof of theorem \ref{eulaszeta}
strongly appeals to the following simple, but crucial, estimate
state in Lemma 4.1 in reference \cite{BVJDE}, to which the reader is
refereed for the proof.
\begin{lemma}
Let $\,\ba \in\, \C_*(\,\Ov)\,$ and $\,\U \in \,
\C^{0,\,\de}(\,\Ov;\,\Ov)\,,$ $\,0<\de\leq\,1\,.$ Then $\,\ba \circ
\,\U \in \, C_*(\,\Ov)\,$; moreover
\begin{equation}
[\,\ba \circ\,\U]_* \leq\,\frac{1}{\de}\,[\,\ba]_*\,.%
\label{aquela}
\end{equation}
\label{columbus}
\end{lemma}%
Note that the need for the above property narrows the possible
choice of the candidate spaces $\,\C_*(\,\Ov)\,$.\par%
The specific role of the $\,\U's\,$ in equation \eqref{aquela} is
that of stream-lines generated by velocity fields. In an old, hand
written version \cite{BVUN} (still conserved) the above lemma is
written in terms of metric spaces. However, at that time, it seemed
to us a little "out of place" to present a so simple result in an
abstract form.\par%
Concerning the $2-D$ Euler equations we also refer the reader to
\cite{koch}, where the treatment of this particular problem is not
very dissimilar to that followed in reference \cite{BVJDE}.
\section{Stokes, and other elliptic problems, in $\,C_*(\Ov)\,$. }\label{doisemeio}%
We start by showing how to confine the minimal regularity problem
for the elliptic system \eqref{naosei2}, treated in Theorem
\ref{rotlaplaces}, to a similar, simpler, regularity problem for
equation \eqref{lapsim}. A classical argument shows that
the solution $\,\bv\,$ of the linear elliptic system \eqref{naosei2} can be obtained by setting%
\begin{equation}
\bv=\,{Rot}\,\psi\,,%
\label{grot}
\end{equation}
where the scalar field $\,\psi\,$ solves the problem
\begin{equation}
\left\{
\begin{array}{l}
-\,\Delta\, \psi=\,\te \quad \textrm{in} \quad \Om \,,\\
\psi=\,0 \quad \textrm{on} \quad \Ga \,. %
\end{array}
\right.%
\label{lapsim}
\end{equation}
We appeal here to a typical approach in studying planar motions. The
scalar $\,\psi\,$ is a simpler representation of
$\,\psi\,\boldsymbol{e}_3\,,$ a vector field normal to the plane of
motion. Furthermore, $\,{Rot}\,\psi\,$ is defined by setting
$\,{Rot}\,\psi\equiv\,{curl}\,(\psi \,\boldsymbol{e}_3)\,,$ a vector
field lying in the plane of motion. It may be obtained by a
$\,\frac{\pi}{2}\,$ counter-clockwise rotation of
$\,\nabla \,\psi\,$.\par%
In reference \cite{BVJDE} we have stated the following result.
\begin{theorem}
Let $\,\te \in \,C_*(\Ov)\,$ and let $\,\psi\,$ be the solution to
problem \eqref{lapsim}. Then $\,\psi \in\, C^2(\Ov)\,,$ moreover,
$\,\|\,\psi\,\|_2 \leq \,c_0\,\|\,\te\,\|_*\,.$%
\label{laplaces}
\end{theorem}
Theorem \ref{rotlaplaces} follows immediately from this result, by
appealing to the explicit expression \eqref{grot} of the solution
$\,\bv\,$ of problem \eqref{naosei2}.\par%
Theorem \ref{laplaces} was stated in \cite{BVJDE} as Theorem 4.5.
Actually, in this last reference the result was claimed for more
general linear elliptic boundary value problems. However the proof
remained unpublished. Recently, in reference \cite{BVSTOKES}, we
followed the same lines to obtain a corresponding result for the
Stokes system, see Theorem \ref{teoum} bellow (in section
\ref{thres} we show a partial extension of this theorem to larger functional spaces).\par%
Consider the Stokes system (see, for instance, \cite{galdi},
\cite{ladyz}, \cite{temam})
\begin{equation}
\left\{
\begin{array}{l}
-\,\De\,\bu+\,\na\,p=\,\ff \quad \textrm{in} \quad \Om \,,\\
\na \cdot\,\bu=\,0  \quad \textrm{in} \quad \Om \,,\\
\bu=\,0 \quad \textrm{on} \quad \Ga \,.%
\end{array}
\right.%
\label{dois}
\end{equation}
If $\,\ff \in\, \bC(\Ov)\,,$ this problem has a unique solution
$\,(\,\bu,\,p\,) \in \bC^1(\Ov)\times\ C(\Ov)\,,$ where $\,p\,$ is
defined up to a constant. The solution is given by
\begin{equation}
u_i(x)=\,\int_{\Om} \,G_{i\,j}(x,\,y) \,f_j(y) \,dy\,, \quad
p(x)=\,\int_{\Om} \,g_j(x,\,y) \,f_j(y) \,dy\,,%
\label{solias}
\end{equation}
where $\,\bG\,$ and $\,\bg\,$ are respectively the Green's tensor
and vector associated with the above boundary value problem.
Furthermore, the following estimates hold.
\begin{equation}
\begin{array}{l}
|\,G_{i\,j}(x,\,y)\,| \leq\,\frac{C}{|x-\,y|}\,,\\
\\
\Big|\, \frac{\pa\,G_{i\,j}(x,\,y)}{\pa\,x_k}\,\Big|
\leq\,\frac{C}{|x-\,y|^2}\,, \quad |\,g_j(x,\,y)\,|
\leq\,\frac{C}{|x-\,y|^2}\,,\\
\\
\Big|\, \frac{\pa^2\,G_{i\,j}(x,\,y)}{\pa\,x_k\,\pa\,x_l}\,\Big|
\leq\,\frac{C}{|x-\,y|^3}\,, \quad |\,
\frac{\pa\,g_j(x,\,y)}{\pa\,x_k}\,| \leq\,\frac{C}{|x-\,y|^3}\,,
\end{array}
\label{quatro}
\end{equation}
where the positive constant $C$ depends only on $\,\Om\,.$ For an
overview on the classical theory of hydrodynamical potentials, and
the construction of the Green functions $\,\bG\,$ and $\,\bg\,,$ we
refer to chapter 3 of the classical treatise \cite{ladyz}. The
estimates \eqref{quatro} are contained in equations (46) and (47) in
this last reference. They may also be found in \cite{Sol60}; see
also \cite{cattabriga} and \cite{valli}. The estimates
\eqref{quatro} are a particular case of a set of much more general
results, due to many authors. See, for instance,
\cite{a-d-n}, \cite{Sol60}, \cite{so2}, \cite{Sol70} and \cite{Sol71}.\par%
It is well known, see \cite{cattabriga}, that for every $\,\ff \in
\,\bC^{0,\,\la}(\Ov)\,$ the solution $(\bv,\,p)$ to the Stokes
system \eqref{dois} belongs to
$\,\bC^{2,\,\la}(\Ov)\times\,C^{1,\,\la}(\Ov)\,$. Hence, as above,
H\"older spaces look too strong as data spaces for getting classical
solutions. On the other hand, as above, it is well known that $\,\ff
\in \,\bC(\Ov)\,$ does not guarantee classical solutions.
In reference \cite{BVSTOKES}, by following \cite{BVUN}, we proved
the following result.
\begin{theorem}
For every $\,\ff \in \,\bC_*(\Ov)\,$ the solution $(\bu,\,p)$ to the
Stokes system \eqref{dois} belongs to
$\,\bC^2(\Ov)\times\,C^1(\Ov)\,$. Moreover, there is a constant
$\,c_0\,$, depending only on $\,\Om\,,$ such that the estimate
\begin{equation}
\|\,\bu\,\|_2 +\,\|\na \,p\,\| \leq \,c_0\,\|\,\ff\,\|_*\,, \quad
\forall \, \ff\in\,\bC_*(\Ov)\,,%
\label{tres}
\end{equation}
holds.%
\label{teoum}
\end{theorem}
A partial generalization of the above theorem to data in a larger
space is presented in section \ref{thres}, see Theorem \ref{prop2}.
\section{Spaces  $\,B_*(\Ov)\,,$ and  $\,D_*(\Ov)\,$. }\label{two}%
The results obtained in the framework of $\,C_*(\Ov)\,$ spaces,
described in sections \ref{conhecidos} and \ref{doisemeio},
immediately lead us to consider their possible extension to larger
functional spaces of continuous functions. Below we consider
functional spaces $\,B_*(\Ov)\,$ and  $\,D_*(\Ov)\,$, such that, in
particular
$$
\,C_*(\Ov) \subset\,\,B_*(\Ov) \subset \,D_*(\Ov)\,.
$$
In \cite{BVUN} we have introduced the space $\,B_*(\Ov)\,$, as
follows. For each $\,f \in\,C(\Ov)\,,$ we define the semi-norm
\begin{equation}%
\langle\,f\,\rangle_* = \,\sup_{ \,x \in\,\Ov}\, \int_0^\ro \,
\om_f(x;\,r)\,
\,\frac{dr}{r}\,,%
\label{seis-bis}
\end{equation}
and the functional space
\begin{equation}
B_*(\Ov) \equiv\,\{\,f \in\,C(\,\Ov): \,\langle\,f\,\rangle_*
\,<\,+\,\infty\,\}%
\label{cstar}
\end{equation}
endowed with the norm
\begin{equation}%
|\|\,f\,\|_* \equiv \, \langle\,f\,\rangle_* +\,\|\,f\,\|\,. %
\label{sete2}
\end{equation}
The reader should compare \eqref{seis-bis} with \eqref{catriz}.
Obviously, $\,\langle\,f\,\rangle_*\leq \,[\,f\,]_*\,$. Actually,
the $\,B_*(\Ov)\,$ norm is "much weaker". In \cite{BVUN} we have
shown that the inclusion $ \,C_*(\Ov) \subset\,B_*(\Ov)\,$ is
proper, by constructing oscillating functions which belong to
$\,B_*(\Ov)\,$ but not to $\,C_*(\Ov)\,.$ This construction was
recently published in reference \cite{BV-LMS}. Further, we may show
that $\,B_*(\Ov)\,$ is compactly embedded in $\,C(\,\Ov)\,,$ and
that \eqref{igual} still holds for the $\,B_*(\Ov)\,$ semi-norm.

\vspace{0.2cm}

In reference \cite{BVUN} we considered the problem \eqref{lapsim}
with data in $\,B_*(\Ov)\,$, and have proved that the the first
order derivatives of  $\,\psi\,$ are Lipschitz continuous in
$\,\Ov\,.$ Hence, second order derivatives are bounded. However we
were (and are) not able to prove the continuity of these derivatives
(the continuity result would hold if the density theorem \ref{lemum}
were to hold with $\,C_*(\Ov)\,$ replaced by $\,B_*(\Ov)\,$; an
interesting open problem). This led us, at that time, to replace in
the published work \cite{BVJDE} the space $\,B_*(\Ov)\,$ by the more
handy space $\,C_*(\Ov)\,$. Actually, in \cite{BVUN}, the result was
proved for linear elliptic boundary value problem whose solutions
are given by
$$
u(x)=\,\int_{\Om} \,G(x,\,y) \,f(y)\,dy\,,
$$
where the \emph{scalar} Green function $\,G(x,\,y)\,$ satisfies the
estimates stated in \eqref{quatro}. Recently, in reference
\cite{BV-LMS}, we have published the proof of this result in a more
general form, since $\,B_*(\Ov)\,$ was replaced by the larger space
$\,D_*(\Ov)\,$, defined below. In section \ref{thres} we appeal to
the same ideas to extend the result proved in\cite{BV-LMS} to the
Stokes problem \eqref{dois}, see Theorem \ref{prop2}. Obviously, all
the results
proved for data in$\,D_*(\Ov)\,$ hold for data in $\, B_*(\Ov)\,.$%

\vspace{0.2cm}

The space $\,D_*(\Ov)\,$ is defined as follows. Set
$$
S(x;\,r)=\,\{\,y \in\, \Ov:\,|y-\,x|=\, r\,\}
$$
and define, for $\,f \in \,C(\Ov)\,$, $\,x_0 \in\,\Ov\,,$ and
$\,r>\,0\,,$ the quantity
\begin{equation}
\om_f(x_0;\,r) \equiv \, \sup_{\,y \in\,S(x_0;\,r)} \,|\,f(y)-\,f(x_0)\,|\,.%
\label{cinco}
\end{equation}
Further, we define the semi-norm
\begin{equation}
(\,f\,)_* \equiv \, \sup_{x_0 \in\,\Ov}\, \int_{0}^{R} \,
\om_f(x_0;\,r) \,\frac{dr}{\,r}\,,%
\label{emef}
\end{equation}
and the related functional space
$$
D_*(\Ov) \equiv\,\{\,f \in\,C(\,\Ov): \,(\,f\,)_*
<\,\infty\,\}\,.%
$$
A norm in $ D_*(\Ov)\,$ is introduced by setting $
\,|\|\,f\,|\|_*=\,(\,f\,)_*+\,\|\,f\,\|\,.$ We remark that
\eqref{igual} still holds for the $\,D_*(\Ov)\,$ semi-norm.

\vspace{0.2cm}

Note that, compared to $\, B_*(\Ov)\,$, the space $\, D_*(\Ov)\,$ is
defined by replacing in \eqref{seis-bis} the expression of
$\,\om_f(x;\,r)\,$ shown in \eqref{vinte} by the expression given by
\eqref{cinco}. Sets $\,\Om(x;\,r)\,$ are replaced by sets
$\,S(x;\,r)\,.$ Also note that $\,S(x;\,r)\,$ is a proper subset of
$\,\pa\,\Om(x;\,r)\,$ if the distance of $\,x\,$ to the boundary
$\,\pa\,\Om\,$ is less than $\,r\,.$\par%
It is worth noting that the above substitution in the definition of
$\,C_*(\Ov)\,$ is irrelevant. It leaves this space invariant.

\vspace{0.2cm}

The results obtained in the framework of $\,C_*(\Ov)\,$ spaces also
lead us to consider the problem of their \emph{restriction} to
smaller functional spaces, instead of extension to larger spaces.
The main motivation, within the realm of solutions to second order
linear elliptic boundary value problems, can be illustrated as
follows. If $\,f \in \, C^{0,\,\la}(\Ov)\,,$ the second order
derivatives of the solution satisfy $\,D^2\,u \in
C^{0,\,\la}(\Ov)\,$. Let's say, for brevity, that they fully
"remember" their origin. On the other hand, if the data $\,f\,$ is
in $\,C_*(\Ov)\,$, then the second order derivatives of the solution
are merely continuous. Roughly speaking, they completely "forget"
that $\,f\,$ produces a finite the integral on the right hand side
of \eqref{seis}. This situation leads us to look for data spaces,
between H\"older and $\,C_*(\Ov)\,$ spaces, for which solutions
"remember", at least partially, their origin. The following is a
significant example of a functional space of "intermediate type".
Define, for each $\,\al>\,0\,,$ the semi-norm
\begin{equation}
[\,f\,]_{0;\al} \equiv\,\sup_{x,\,y \in\,\Ov \quad
0<\,|x-\,y|<\,1}\,\frac{|f(x)-\,f(y)\,|}{(-\log{|\,x-\,y|})^{-\,\al}}
\,,%
\label{alfas2}
\end{equation}
and the related norm $
\,\|\,f\,\|_{0;\al}\equiv\,[\,f\,]_{0;\al}+\,\|\,f\,\|\,.$ Next,
define functional spaces $\,D^{0,\,\al}(\Ov)\,$ in the obvious way.
Roughly speaking, we have replaced in the definition of H\"older
spaces the quantity
$$
\frac{1}{|\,x-\,y|} \quad \textrm{ by}  \quad
\log{\frac{1}{|\,x-\,y|}}\,.
$$
This similitude leads us to call these spaces H\"olderlog (H\"older-
logarithmic) spaces. The family of H\"olderlog spaces enjoys some
typical, significant property. For instance, $\,D^{0,\,\al}(\Ov)\,$
is a Banach space, and $\,C^{\infty}(\Ov)\,$ is a dense subspace.
Furthermore, for $\,0<\,\bt <\,1<\,\al\,,$ and $\,0<\,\la
\leq\,1\,,$ the following strict embeddings
\begin{equation}
\quad C^{0,\,\la}(\Ov)\subset D^{0,\,\al}(\Ov) \subset
\,C_*(\Ov)\,\subset D^{0,\,\beta}(\Ov)\, \subset \,C(\Ov)%
\label{oitooos}
\end{equation}
hold. Note that $\,D^{0,\,1}(\Ov) \subset \,C_*(\Ov)\,$ is false.
The embeddings $\, D^{0,\,\al}(\Ov)\subset D^{0,\,\bt}(\Ov) \subset
\,C(\Ov)\,,$ for $\,\al>\bt>0\,,$ and the embeddings
$\,D^{0,\,\al}(\Ov) \subset\, C_*(\Ov)\,,$ for $\,\al >\,1\,,$ are
compact.\par%
In forthcoming papers we consider boundary value problems with data
in $\, D^{0,\,\al}(\Ov)\,$. For a second order linear elliptic
problem we show, in reference \cite{BV-JP}, that if $\,f \in
\,D^{0,\,\al}(\Ov)\,$, for some $\,\al>\,1\,,$ then $\,D^2\,\bu \in
D^{0,\,(\al-1)}(\Ov)\,.$ There is just a "partial loss of
regularity". Full application to the Euler problem \eqref{alho},
will be also shown.
\section{The Stokes equations with data in $\,D_*(\Ov)\,.$ Uniform boundedness of $\,\na^2\,\bu\,$ and $\,\na\,p\,$
.}\label{thres}%
In this section we consider the Stokes system and show that the
first order derivatives of the velocity $\,\bu\,,$ and the pressure
$\,p\,,$ are Lipschitz continuous in $\,\Ov\,$ for given external
forces in $\,D_*(\Ov)\,$ (so, in particular, in $\,B_*(\Ov)\,$). We
prove the following result.
\begin{theorem}
Let $\,\ff \in \,\bD_*(\Ov)\,,$ and let $\,(\bu,\,p)\,$ be the
solution to problem \eqref{dois}. There is a constant $\,C\,$, which
depends only on $\,\Om\,,$ such that
\begin{equation}
\|\,\bu\,\|_{1,\,1}+\,\|\,p\,\|_{0,\,1} \leq\,C\, |\|\,\ff\,|\|_*\,.%
\label{treze}
\end{equation}%
So $\,\na^2\,\bu\,,\na\,p \in\,L^{\infty}(\Om)\,$. 
\label{prop2}
\end{theorem}
\begin{proof}
In the following we merely consider the velocity, since the pressure
is treated similarly (see also \cite{BVSTOKES}). Let
$\,\boldsymbol{e}_i(x)\,,$ $\,i=\,1,\,2,\,3\,,$ denote three
constant vector fields in $\,\R^3\,$, everywhere equal to the
corresponding cartesian coordinate unit vector
$\,\boldsymbol{e}_i\,$. Define the auxiliary systems
\begin{equation}
\left\{
\begin{array}{l}
-\,\De\,\bv_i(x)+\,\na\,q_i(x)=\,\boldsymbol{e}_i(x) \quad \textrm{in} \quad \Om \,,\\
\na \cdot\,\bv_i=\,0  \quad \textrm{in} \quad \Om \,,\\
\bv_i=\,0 \quad \textrm{on} \quad \Ga \,.%
\end{array}
\right.%
\label{doiszetas}
\end{equation}
Clearly, $\,\bv_i\,$ and $\,q_i\,$ are smooth. Fix a constant
$\,K(\Om)\,$ such that
\begin{equation}
\|\,\bv_i\,\|_{1,\,1} +\,\|\,q_i\,\|_{0,\,1} \leq\,K(\Om)\,,%
\label{quinze}
\end{equation}
for $\,i=\,1,\,2,\,3\,.$ Next, in correspondence to each point
$\,x_0 \in \,\Om\,,$ define the auxiliary system (a kind of "tangent
problem" at point $x_0\,$)
\begin{equation}
\left\{
\begin{array}{l}
-\,\De\,\bv(x_0,\,x)+\,\na\,q(x_0,\,x)=\, \ff(x_0,\,x) \quad \textrm{in} \quad \Om \,,\\
\na \cdot\,\bv=\,0  \quad \textrm{in} \quad \Om \,,\\
\bv=\,0 \quad \textrm{on} \quad \Ga \,,%
\end{array}
\right.%
\label{doiszero}
\end{equation}
where $\, \ff(x_0,\,x)\equiv\,\ff(x_0)\,,\forall \,x\in \,\Om\,,$ is
a constant vector in $\,\Om\,.$ Since
$$
\ff(x_0,\,x)=\,\sum_{i} \,f_i(x_0)\, \,\boldsymbol{e}_i(x)\,,
$$
the functions $\,\bv(x_0,\,x)\,$ and $\,q(x_0,\,x)\,$ are smooth for
each fixed $\,x_0\,$. Moreover,
\begin{equation}
\|\,\bv(x_0,\,\cdot)\,\|_{1,\,1} +\,\|\,q(x_0,\,\cdot)\,\|_{0,\,1}
\leq\,K\,|\ff(x_0)\,|\leq\,K\,\|\,\ff\,\|\,.%
\label{quinze}
\end{equation}
Recall that $\,K\,$ is independent of $\,x_0\,.$ For convenience set
$\, \bv(x)=\,\bv(x_0,\,x)\,,$ and so on. By setting
$$
\bw(x)\equiv\,\bu(x)-\,\bv(x)\,,
$$
one has
$$
w_i(x)=\,\int_{\Om} \,G_{i\,j}(x,\,y)\,(\,f_j(y)-\,f_j(x_0)\,) \,dy
\,.
$$
Furthermore,
$$
\pa_k\,w_i(x)-\,\pa_k\,w_i(x_0)=\,\int_{\Om}
\,\big(\,\pa_k\,G_{i\,j}(x,\,y)-\,\pa_k\,G_{i\,j}(x_0,\,y)\,\big)
\,(\,f_j(y)-\,f_j(x_0)\,) \,dy \,,
$$
where $\,\pa_k\,$ stands for differentiation with respect to
$\,x_k\,,$  and $\,\pa_k\,w_i(x_0)\,$ means the value of
$\,\pa_k\,w_i(x)\,$ at the particular point $\,x=\,x_0\,.$\par%
Clearly
$$
|\,\pa_k\,w_i(x)-\,\pa_k\,w_i(x_0)\,|\leq\,\int_{\Om} \,
|\,\pa_k\,G_{i\,j}(x,\,y)-\,\pa_k\,G_{i\,j}(x_0,\,y)\, |
\,|\,f_j(y)-\,f_j(x_0)\,| \,dy\,.
$$
By setting $\,\rho =\,|x-\,x_0|\,$ one gets
\begin{equation}
\begin{array}{l}
|\,\pa_k\,w_i(x)-\,\pa_k\,w_i(x_0)\,|\leq \\
\\
\,\int_{\Om(x_0;\,2\,\rho)}|\,\pa_k\,G_{i\,j}(x,\,y)-\,\pa_k\,G_{i\,j}(x_0,\,y)\,
|
\,|\,f_j(y)-\,f_j(x_0)\,| \,dy+\\
\\
\int_{\Om_c(x_0;\,2\,\rho)}\,
|\,\pa_k\,G_{i\,j}(x,\,y)-\,\pa_k\,G_{i\,j}(x_0,\,y)\,|
\,|\,f_j(y)-\,f_j(x_0)\,|\, dy \\
\\
\equiv\,I_1(x_0,\,x,\,\ro) +\,I_2(x_0,\,x,\,\ro)\,.
\end{array}
\label{dezaseis}
\end{equation}
By appealing to \eqref{quatro}, we show that
\begin{equation}
\begin{array}{l}
I_1(x_0,\,x,\,\ro) \leq\,C \,\Big(\,
\int_{\Om(x_0;\,2\,\rho)}\,\frac{C}{|\,x_0-\,y\,|^2}
\,|\,f_j(y)-\,f_j(x_0)\,|dy+\\
\\
\int_{\Om(x;\,3\,\rho)}\, \frac{C}{|x-\,y|^2}\,
|\,f_j(y)-\,f_j(x_0)\,| \,dy\,\Big)\, \equiv\,J_1(x_0,\,x,\,\ro)
+\,J_2(x_0,\,x,\,\ro)\,.
\end{array}
\label{dezzaseis-n}
\end{equation}
By setting $\,r=\,|x_0-\,y\,|\,,$  and by appealing to
polar-spherical coordinates centered in $\,x_0\,$, one easily shows
that $\, J_1(x_0,\,x,\,\ro)
\leq\,C\,\ro\,\om(\,\Om(x_0;\,2\,\rho\,)\,)\,,$ where $C$ depends
only on $\,\Om\,.$ Similarly, $\,J_2(x_0,\,x,\,\ro)
\leq\,C\,\ro\,\om(\,\Om(x;\,3\,\rho\,)\,)\,.$ It follows that
(recall definition \eqref{cinco})
$$
I_1(x_0,\,x,\,\ro) \leq\,C\,\ro\,\om_{\ff}(3\,\ro)\,,
$$
where $C$ does not depend on the particular points $\,x_0,\,x \in
\,\Ov\,,$ and $\,\rho =\,|x-\,x_0|\,.$\par%
On the other hand, by appealing to the mean-value theorem and to
\eqref{quatro}, we get
$$
|\,\pa_k\,G_{i\,j}(x,\,y)-\,\pa_k\,G_{i\,j}(x_0,\,y)\, |
\leq\,C\,\rho\,|x'-\,y|^{\,-3}\leq\,
C\,\rho\,2^3\,|x_0-\,y|^{\,-3}\,,
$$
for each $\,y\in\,\Om_c (x_0;\,2\,\rho)\,,$ where the point $x'$
belongs to the straight segment joining $x_0$ to $x$. Consequently,

$$
I_2(x_0,\,x,\,\ro) \leq \, c\,\rho \,
\int_{\Om_c(x_0;\,2\,\rho)}\,|\,f_j(y)-\,f_j(x_0)\,|\,
\frac{dy}{\,|x_0-\,y|^3} \, \leq\,c\,\rho \,
\int_{2\,\rho}^{R}\,\om_{\ff}(r)\,\frac{dr}{r}\,.
$$
Hence,
$$
I_2(x_0,\,x,\,\ro) \leq\,c\,\rho \,
\int_{0}^{R}\,\om_{\ff}(r)\,\frac{dr}{r}=\,\,c\,\rho
\,(\,\ff\,)_*\,.
$$
Next, by appealing to equation \eqref{dezaseis}, and to the
estimates proved above for $I_1$ and $I_2\,,$ we show that
$$
|\,\na\,\bw(x)-\,\na\,\bw(x_0)\,| \leq\,c\,\rho\,(
\,(\,\ff\,)_*+\,\om_{\ff}(3\,\ro)\,)\,.
$$
Consequently,
$$
\begin{array}{l}
|\,\na\bu(x)-\,\na\,\bu(x_0)| \leq
\,|\,\na\bw(x)-\,\na\,\bw(x_0)\,|+\,
|\,\na\,\bv(x)-\,\na\,\bv(x_0)\,|\\
\\
\leq\,c\,\rho\,\big( \,(\,\ff\,)_* +\,\om_{\ff}(3\,\ro)\,+
K\,\|\,\ff\,\| \,\big)\,.
\end{array}
$$
So,
\begin{equation}
\frac{|\,\na\,\bu(x)-\,\na\,\bu(x_0)|}{|x-\,x_0|}
\leq\,C\,\||\,\ff\,|\|_*\,,
\quad \forall x,\,x_0 \in \Om, \,x \neq\,x_0\,.%
\label{dezasete}
\end{equation}
This proves \eqref{treze} for the velocity $\,\bu\,.$ Similar
calculations lead to the corresponding result for the pressure.
\end{proof}
It is worth noting that our proofs depend only on having suitable
estimates for the Green's functions. For instance, the argument
applied in the above proof to study the system \eqref{dois} with
data in $\,D_*(\Ov)\,$ can be applied to the system \eqref{lapsim}
with data in $\,D_*(\Ov)\,,$ since the scalar Green's function
$\,G(x,\,y)\,$ related to this last problem satisfies exactly the
estimates claimed in equation \eqref{quatro} for the components
$\,G_{i\,j}(x,\,y)\,$. It follows that Theorem \ref{laplaces} still
holds in the above "weak form", for data in $\,D_*(\Ov)\,.$ Further,
in correspondence with Theorem \ref{rotlaplaces}, we get the
following statement.
\begin{theorem}
Let $\,\te \in \,D_*(\Ov)\,,$ and let $\,\bv\,$ be the solution of
problem \eqref{naosei2}. Then $\,\na\,\bv \in\, L^{\infty}(\Ov)\,$,
and $\,\|\,\na\,\bv\,\|_{L^{\infty}(\Ov)} \leq
\,c_0\,\||\,\te\,|\|_*\,.$ So, for divergence free vector fields,
tangent to the boundary, the estimate%
\begin{equation}
\|\,\na\,\bv\,\|_{L^{\infty}(\Ov)} \leq
\,c_0\,\||\,{curl}\,\bu\,|\|_*\,%
\label{dezset}
\end{equation}
holds.%
\label{lepasb}
\end{theorem}
This specific case will be useful in considering the Euler equations
with data in $\,B_*(\Ov)\,.$ For more results and comments on the
above subject we refer to \cite{BV-LMS}.


\vspace{0.3cm}

\section{The space $\, B_*(\Ov)\,$ and the Euler equations. }%
\label{bbesp}
Concerning possible extensions of the results obtained for the $2-D$
evolution Euler equations, from $\,\C_*(\Ov)\,$ to $\,\bB_*(\Ov)\,,$
we show here a partial result in this direction. We clearly pay the
price of the loss of regularity for solutions to the auxiliary
elliptic system \eqref{naosei2}. This leads us to
replace continuity in time by boundedness in time.\par%
Furthermore, we simplify our task, in a \emph{quite substantial
way}, by assuming that external forces vanish, instead of assuming
the very stringent condition $\,{curl}\,\ff \in \,L^1(\,R^+;\,
B_*(\Ov)\,)\,$.\par%
Below, we prove the following weak extension of theorem
\ref{eulaszeta}.
\begin{theorem}
Let $\,\bv\,$ be the solution to the Euler equations \eqref{alho},
where the initial data $\,\bv_0\,$ is divergence free, tangential to
the boundary, and satisfies $\,{curl}\,\bv_0 \in \,B_*(\Ov)\,.$
Furthermore, suppose $\,\ff=\,0\,$. Then $\,{curl}\,\bv
\in\,L^{\infty}(\,0,\,T;\,B_*(\Ov)\,)\,$, and there is a constant
$\,C_T\,$ (an explicit expression can be easily obtained) such that
\begin{equation}
|\|\,{curl}\,\bv(t)\,\|_* \leq \, C_T\,|\|\,{curl}\,\bv_0\,\|_{*} \,,%
\label{estmeulerzeta2}
\end{equation}
for a.a. $\,t \in\,(0,\,T)\,.$
\label{zeulerin}%
\end{theorem}
A weak extension of theorem \ref{eulas} follows immediately from
theorem \ref{zeulerin} together with Theorem \ref{lepasb}. One has
the following result.
\begin{theorem}
Under the assumptions of theorem \ref{zeulerin} the estimate
\begin{equation}
\|\,\nabla\,\bv\,\|_{L^{\infty}(Q_T)}
\leq\,C_T\,|\|\,{curl}\,\bv_0\,\|_{*}%
\label{istase}
\end{equation}
holds almost everywhere in $\,Q_T\,.$%
\label{zeulerindos}
\end{theorem}
To prove Theorem \ref{zeulerin}, we appeal to some estimates
previously obtained in a more general form in reference
\cite{BVJDE}. For clarity, instead of stating these estimates in the
weakest form, strictly necessary to prove the theorem \ref{zeulerin}
below, we rather prefer to show some more general formulations of
the estimates. This allows us to present a short overview on the
structure of the proof of theorem \ref{estmeuler}, suitable for
readers interested in a deeper examination of reference
\cite{BVJDE}. In order to make an easier link with this last
reference, we appeal here to the notation used in \cite{BVJDE}
(compare, for instance, \eqref{naosei} and \eqref{lapsim2} below
with \eqref{naosei2} and \eqref{lapsim}, respectively).%

\vspace{0.2cm}

As already remarked, the velocity $\,\bv(t)\,,$ at each time
$\,t\,,$ can be obtained from the vorticity $\,\z(t)\equiv\,{curl}
\,\bv(t)\,$, by setting, for each fixed $\,t\,,$  $\,\te=\,\z(t)\,$
in the elliptic system
\begin{equation}
\left\{
\begin{array}{l}
{curl} \,\bv=\,\te \quad \textrm{in} \quad \Om \,,
\\
div\,\bv=\,0 \quad \textrm{in} \quad \Om \,,
\\
{curl}
\,\bv \cdot\,\n=\,0 \quad \textrm{on} \quad \Ga \,. %
\end{array}
\right.%
\label{naosei}
\end{equation}
On the other hand, the solution to this system is given by
\begin{equation}
\bv=\,{Rot}\,\ps\,,%
\label{fict}
\end{equation}
where $\ps\,$ is the solution of the elliptic problem
\begin{equation}
\left\{
\begin{array}{l}
-\,\Delta\, \,\ps=\,\te \quad \textrm{in} \quad \Om \,,\\
\ps=\,0 \quad \textrm{on} \quad \Ga \,. %
\end{array}
\right.%
\label{lapsim2}
\end{equation}
So, at least in principle, we may obtain the velocity from the
vorticity. However, since the vorticity is not a priori known, we
start from a "fictitious vorticity" $\,\te(x,\,t)\,,$ and look for a
fixed point $\,\te=\,\z\,.$ In the sequel we replace "fictitious
vorticity" simply by "vorticity", and so on for other quantities.
From each suitable "vorticity" we obtain a "velocity", by appealing
to \eqref{lapsim2} and \eqref{fict}. From this "velocity" we
construct streamlines $\,\U(s,\,t,\,x)\,,$ by appealing to
Lagrangian coordinates. Finally, a well know technique (here
dimension $\,2\,$ is crucial) gives a correspondent fictitious
"vorticity" $\,\z\,.$ So, a map $\,\te \rightarrow \,\z\,$ is,
formally, well defined. A rigorous fixed point was obtained in
reference \cite{BVJDE} in the framework of $C(\Ov)\,$ spaces,
as follows:%

\vspace{0.2cm}

Fix an arbitrary positive time $\,T\,,$ an initial data $\,\bv_0\,$,
and an external force $\,\ff\,$. Set $\,\z_0\equiv\,{curl}
\,\bv_0\,$, $\,\phi\equiv\,{curl}\,\ff\,,$ and define (see
\eqref{bbeesze})
\begin{equation}
B=\,\|\,\z_0\,\| +\,\int_0^T \, \|\,\phi(\ta)\,\| \,d\ta\,.%
\label{zeqa}
\end{equation}%
Further, define the convex, bounded, closed subset of
$\,C(\overline{Q}_T)\,$,
\begin{equation}%
\K=\,\{\,\te \in C(\overline{Q}_T):\,\|\,\te\,\|_T
\leq\,B\,\}\,.%
\label{uiui}
\end{equation}
From now on, the symbol $\,\te=\,\te(x,\,t)\,$ denotes an arbitrary
element of $\,\K\,$. As already explained, the idea is to prove the
existence and uniqueness of a fixed point in $\,\K\,,$ for a
suitable map $\,\Pf\,,$ such that to this fixed point there
corresponds a solution of the Euler equation \eqref{alho} with the
above given data. The map $\,\Pf[\te]=\z\,$ is defined as the
following composition of single maps:
\begin{equation}
\Pf:\quad \te \rightarrow\,\ps \rightarrow\,\bv \rightarrow\,\U
\rightarrow\,\z\,.%
\label{setass}
\end{equation}
Given $\,\te=\,\te(x,\,t) \in\,\K\,$ we get $\,\ps=\,\ps(x,\,t)\,$
by solving the elliptic system \eqref{lapsim2}, where $\,t\,$ is
treated as a parameter. The crucial estimates for $\,\ps(x)\,$
follow from
$$
\ps(x)=\,\int_{\Om} \,g(x,\,y) \,dy,
$$
where $\,g\,$ is the Green function associated to problem
\eqref{lapsim2}. Knowing $\,\ps\,,$ the velocity $\,\bv\,$ is
obtained
by setting $\,\bv(x,\,t)=\,{Rot}\, \ps(x,\,t)\,.$\par%
The next step is to get $\,\z\,,$ from $\,\bv\,$. We introduce the
streamlines $\,\U\,$ associated with the "velocity" $\,\bv(x,\,t)\,$
obtained in the previous step. The streamlines $\,\U(s,\,t,\,x)\,$
are the solution to the system of ordinary differential equations
\begin{equation}
\left\{
\begin{array}{l}
\frac{d}{d\,s}\,\U(s,\,t,\,x)=\,\bv(\,s,\,\U(s,\,t,\,x)\,)\,, \quad
\textrm{for} \quad s\in\,[0,\,T\,]\,,\\
\\
\U(t,\,t,\,x)=\,x\,.
\end{array}
\right.%
\label{uus}
\end{equation}
$\quad\U(s,\,t,\,x)\,$ denotes the position at time $\,s\,$ of the
physical particle which occupies the position $\,x\,$ at time
$\,t\,$. A main tool is here the following estimate (see equation
(2.6) in \cite{BVJDE}).
\begin{equation}
\begin{array}{l}
|\,\U(\,s,\,t,\,x)-\,\U(\,s_1,\,t_1,\,x_1)\,| \leq\\
\\
c_1\,B\,|s-\,s_1| +c_2\,(1+\,c_1\,B\,)\,(\,|x-\,x_1|^\de
+\,|t-\,t_1|^\de\,)\,,
\end{array}
\label{dois6}
\end{equation}%
where $\,c_1\,$ depends only on $\,\Om\,$, $\,\de\equiv
\,e^{-\,c_1\,B\,T}\,,$ and $\,c_2=\,{max}\{1,\,e\,R\}\,,$ where
$\,R$ denotes the diameter of $\,\Om\,.$ Knowing $\,\U\,,$ we set
(\cite{BVJDE}, equation (2.8))
\begin{equation}
\begin{array}{l}
\,\z(t,\,x) =\,\z_0(\U(\,0,\,t,\,x)\,) +\,\int_0^t
\,\pf(\,s,\,\U(s,\,t,\,x)\,) \,ds\\
\\
\equiv\,\z_1(t,\,x)+\,\z_2(t,\,x)\,,%
\end{array}
\label{corlas}
\end{equation}%
where, as already remarked, $ \z_0\,\equiv \,{curl}\,\bv_0\,,$ and $
\,\pf \equiv \,{curl}\,\ff\,.$ The curl of the solution is here
expressed separately in terms of the curls of the initial data and
of the external forces. The main estimates for these two terms were
proved in \cite{BVJDE}, respectively in lemmas 4.3 and 4.2. The
reader may verify that the control of the external forces term is
much more involved than that of the initial data term.\par%
The composition map $\,\Pf\,[\te]=\z\,$ turns out to be well defined
over $\,\K\,,$ by appealing to \eqref{setass}. In the proof of
theorem \ref{ppum} in \cite{BVJDE}, we close the above scheme by
showing that $\,\Pf\,(\,\K\,)\subset\,\K\,,$ and that there is a
(unique) fixed point in $\,\K\,.$ Finally, it was proved that this
fixed point is the curl of the solution to the Euler equations
\eqref{alho}. The velocity follows from the curl by appealing to
\eqref{naosei}.\par%
After this flying visit to the proof of theorem \ref{ppum}, we prove
the theorem \ref{zeulerin}.

\vspace{0.2cm}

\emph{Proof of theorem  \ref{zeulerin}}: A main tool in proving the
regularity theorem \ref{eulaszeta} for data in $C_*(\Ov)\,$ is the
lemma \ref{columbus} (see the Lemma 4.1 in \cite{BVJDE}). The
absence of external forces $\,\ff\,$ lead us to revive below,
directly, the simple idea used in the proof of this lemma, without
appealing to the original statement itself.\par%
We deal with solutions whose existence is already guaranteed by
theorem \ref{ppum}. We merely  want to show the additional
regularity claimed in theorem \ref{zeulerin}. Since in this theorem
the external forces vanish, the following very simplified form of
\eqref{dois6} holds.
\begin{equation}
|\,\U(0,\,t,\,x) -\,\U(0,\,t,\,y)\,| \leq\,K \,|x-\,y|^{\de}\,,
\label{uus}
\end{equation}%
where here $\de=\,e^{-\,c_1\,B\,T}\,,$ and $\,B=\,\|\,\z_0\,\|\,$.
Following \eqref{corlas}, and taking into account that $\,\z_2\,$
vanishes, one has $\,\z=\,\z_1\,.$ So the curl of the solution
$\,\bv\,$ to the Euler equation \eqref{alho} is simply given by
$$
\z(t,\,x)=\,\z_0 ( \U(0,\,t,\,x)\,)\,.
$$
It follows that
\begin{equation}
\begin{array}{l}
\om_{\z(t)}(x;\,r)=\, \sup_{y \in\,\Ov(x;\,r)}\, |\,\z
(t,\,x)-\,\z_1(t,\,y)\,|=\\
\\
\sup_{y \in\,\Ov(x;\,r)}\,|\,\z_0 (
\U(0,\,t,\,x)\,)-\,\z_0 ( \U(0,\,t,\,y)\,)\,|\,.%
\end{array}%
\label{zuus}
\end{equation}%
Further, by appealing to \eqref{uus}, one gets
\begin{equation}
\om_{\z(t)}(x;\,r) \leq\,\om_{\z_0}(\U(0,\,t,\,x);\,K\,r^\de)\,.
\label{zuava}
\end{equation}%
So, by recalling definition \eqref{seis-bis}, one has
\begin{equation}
\langle\,\z(t)\,\rangle_* \equiv\,\sup_{x \in\,\Ov}\,\int_0^\ro \,
\om_{\z(t)}(x;\,r) \frac{d\,r}{r} \leq\,\sup_{x \in\,\Ov
}\,\int_0^{\ro}
\om_{\z_0}(\U(0,\,t,\,x);\,K\,r^\de)\,  \frac{d\,r}{r}\,.%
\label{zuava}
\end{equation}
Since
$$
\big\{\,\U(0,\,t,\,x):\,x\in\,\Ov\,\big\}=\,\Ov\,,
$$
it follows, by appealing to the change of variables
$\,\ta=\,K\,r^\de\,,$ that
\begin{equation}
\begin{array}{l}
\langle\,\z(t)\,\rangle_* \leq\,\sup_{\ox \in\,\Ov }\,\int_0^{\ro}
\om_{\z_0}(\ox;\,K\,r^\de) \, \frac{d\,r}{r}=\\
\\
\frac{1}{\de}\,\sup_{x \in\,\Ov }\,\int_0^{K\,\ro^{\de}}
\om_{\z_0}(x;\,\ta)\, \frac{d\,\ta}{\ta} \equiv\, \frac{1}{\de}
\,\langle\,\z_0\,\rangle_{*,\,K\,\ro^{\de}}\,,%
\label{zic}
\end{array}
\end{equation}
with obvious notation. On the other hand,
$\,\|\,\z(t)\,\|=\,\,\|\,\z_0\,\|\,,$ for all $\,t\,.$ Since
\eqref{igual} also applies for $\,B_*\,$ semi-norms, one shows that
$$
|\|\,\z(t)\,\|_{*} \leq\,C_T\,|\|\,\z_0\,\|_{*}\,,
$$
for all $\,t \in\,[0,\,T]\,.$ Theorem \ref{zeulerin} is proved.
Theorem \ref{zeulerindos} follows by appealing to Theorem \ref{lepasb}\,.\par%

\vspace{0.2cm}

It would be interesting to prove theorem \ref{zeulerin} in the
presence of external forces, even in a simplified version, for
instance, $\,{curl}\,\ff \in \,C(\,\R^+;\, B_*(\Ov)\,)\,$. We
believe that a (possibly modified) version of this result holds by
appealing to the measure preserving properties of the streamlines,
together with the control of the linear dimensions of figures in
finite time.

\end{document}